\documentclass[11pt]{amsart}
\usepackage{hyperref}
\usepackage{anysize} \marginsize{1.3in}{1.3in}{1in}{1in}
\usepackage{amsfonts}

\let\oldtocsection=\tocsection

\let\oldtocsubsection=\tocsubsection

\let\oldtocsubsubsection=\tocsubsubsection

\renewcommand{\tocsection}[2]{\bfseries\hspace{0em}\oldtocsection{#1}{#2}}
\renewcommand{\tocsubsection}[2]{\hspace{1em}\oldtocsubsection{#1}{#2}}
\renewcommand{\tocsubsubsection}[2]{\itshape\hspace{2em}\oldtocsubsubsection{#1}{#2}}

\usepackage{comment}
\usepackage[all,cmtip]{xy}
\usepackage{amsmath}
\usepackage{enumitem}
\usepackage{amsthm}
\usepackage{amssymb}
\usepackage{mathrsfs}
\usepackage{enumitem}
\usepackage{tikz-cd}
\usepackage{tikz}
\usepackage[normalem]{ulem}

\newtheorem{thm}{Theorem}[section]
\newtheorem{conj}{Conjecture}[section]
\newtheorem{prop}[thm]{Proposition}
\newtheorem{lemma}[thm]{Lemma}

\newtheorem{cor}[thm]{Corollary}

\theoremstyle{definition}
\newtheorem{defn}[thm]{Definition}

\newtheorem*{definition*}         {Definition}

\newtheorem{eg}[thm]{Example}
\theoremstyle{remark}
\newtheorem{claim}[thm]{Claim}
\newtheorem{remark}[thm]{Remark}

\newcommand{\End}{\mathrm{End}}
\newcommand{\Hom}{\mathrm{Hom}}

\newcommand{\bP}{\mathbb{P}}
\newcommand{\bE}{\mathbb{E}}

\newcommand{\Z}{\mathbb{Z}}

\newcommand{\C}{\mathbb{C}}

\newcommand{\bG}{\mathbb{G}}

\newcommand{\cO}{\mathcal{O}}
\newcommand{\cD}{\mathcal{D}}
\newcommand{\cE}{\mathcal{E}}
\newcommand{\cF}{\mathcal{F}}
\newcommand{\cG}{\mathcal{G}}

\newcommand{\cL}{\mathcal{L}}

\newcommand{\cS}{\mathcal{S}}
\newcommand{\cU}{\mathcal{U}}
\newcommand{\cV}{\mathcal{V}}
\newcommand{\cW}{\mathcal{W}}

\newcommand{\bfG}{\mathbf{G}}
\newcommand{\bfH}{\mathbf{H}}
\newcommand{\bfU}{\mathbf{U}}
\newcommand{\bfGL}{\mathbf{GL}}

\renewcommand{\phi}{\varphi}
\def\ra{\rightarrow}
\def\codim{\operatorname{codim}}
\def\img{\operatorname{img}}

\def\id{\operatorname{id}}
\def\GL{\mathbf{GL}}
\def\SL{\mathbf{SL}}

\def\Hom{\operatorname{Hom}}

\def\Stab{\operatorname{Stab}}

\def\End{\operatorname{End}}

\def\Pairs{\mathrm{Pairs}}
\def\eff{\mathrm{eff}}
\def\mot{\mathrm{mot}}
\def\trdeg{\operatorname{trdeg}}
\def\NoriMot{\mathcal{MM}_{\mathrm{Nori}}}

\def\Spec{\operatorname{Spec}}

\def\gr{\operatorname{gr}}

\def\Zar{\mathrm{Zar}}
\def\colim{\operatorname{colim}}

\def\Hilb{\operatorname{Hilb}}

\newcommand{\alg}{\mathrm{alg}}

\newcommand{\an}{\mathrm{an}}

\renewcommand{\bar}[1]{\overline{#1}}

\usepackage{amscd,amssymb,amsmath}
\usepackage{color}

\title[Functional Transcendence of Periods]{Functional Transcendence of Periods and the Geometric Andr\'e--Grothendieck Period Conjecture}

 \author[B. Bakker]{Benjamin Bakker}
\address{\noindent B. Bakker:  Dept. of Mathematics, Statistics, and Computer Science, University of Illinois at Chicago, Chicago, USA.}
\email{bakker.uic@gmail.com}

\author[J. Tsimerman]{Jacob Tsimerman}
\address{\noindent J. Tsimerman:  Dept. of Mathematics, University of Toronto, Toronto, Canada.}
\email{jacobt@math.toronto.edu}

\def\G{\mathbf{G}}
\def\H{\mathbf{H}}
\def\O{\mathcal{O}}
\def\VV{\mathbb{V}}
\def\D{\mathcal{D}}
\def\V{\mathcal{V}}
\def\R{\mathbb{R}}
\def\Q{\mathbb{Q}}
\def\an{\mathrm{an}}
\def\dual{\vee}
\begin{document}
\maketitle
\begin{abstract}
    We prove a functional transcendence theorem for the integrals of algebraic forms in families of algebraic varieties. This allows us to prove a geometric version of Andr\'e's generalization of the Grothendieck period conjecture, which we state using the formalism of Nori motives.  
    
    More precisely, we prove a version of the Ax--Schanuel conjecture for the comparison between the flat and algebraic coordinates of an arbitrary admissible graded polarizable variation of integral mixed Hodge structures.  This can be seen as a generalization of the recent Ax--Schanuel theorems of \cite{chiu,GaoKlingler} for mixed period maps. 
\end{abstract}

\tableofcontents
\newpage
\section{Introduction}

\subsection{Transcendence of periods}

\subsubsection{Periods}

Given a smooth algebraic variety $X$ defined over a field $k\subset \C$, we may take an algebraic differential $p$-form $\omega$ defined over $k$ (or more generally a degree $p$ algebraic de Rham cohomology class defined over $k$) and form the integral $\int_\gamma\omega$ along a topological $p$-cycle $\gamma$ of $X(\C)$ (with the euclidean topology).  Such numbers are called \emph{periods} of $H^p(X)$.  They collectively determine the Hodge structure on the degree $p$ cohomology of $X$, and conjecturally encode much of the geometry of $X$.  One instance of this is the Hodge conjecture, which says that Hodge classes are represented by algebraic subvarieties.  In a different direction, the period conjecture says that algebraic relations among the periods themselves should be of geometric origin.

\subsubsection{Classical period conjecture}We may measure the $\Q$-algebraic independence of the periods of $H^p(X)$ via the transcendence degree of the field extension of $\Q$ they generate.  The algebraic relations arising from geometry are measured by the motivic Galois group\footnote{We use the Tannakian category of Nori motives to define the motivic Galois group.  Ayoub uses Voevodsky's triangualated category of motives, but the resulting motivic Galois group is canonically the same \cite{isomgalois}.} $\bfG_\mot(H^p(X))$ of $H^p(X)$.  Roughly speaking, it records the $\Q$-algebraic relations arising from pushing forward or pulling back along geometric maps and applying Stokes' theorem.  The two are related by a conjecture of Grothendieck (see \cite{BoCh} for more history and discussion) which is phrased more generally for motives, but for the purposes of the above discussion we may take $M=H^p(X)$ for a variety $X$ defined over $\Q$:




\begin{conj}[Grothendieck period conjecture]\label{Grothperiod}
Let $M$ a Nori motive over $\Q$.  Then
\[\trdeg_\Q \Q(\mathrm{periods}\;\mathrm{of}\;M) = \dim\bfG_\mot(M).\]
\end{conj}

This conjecture is important partly because many interesting numbers arise as periods, such as $\log 2, \pi$.

\subsubsection{Andr\'e's generalization}

While Grothendieck's conjecture is very general (and currently very, very open) it does not cover many important situations. For example even $e$ itself is (conjecturally!) not a period. However, $e$ can still be described in the language of periods, as $\log e=1$ and $\log x$ is itself a period \emph{function}. 

Andr\'e \cite{bertolinappendix} found a very clever way to address this issue by considering periods of varieties not just over number fields, but of varieties over arbitrary subfields $K$ of $\C$. Of course, one must now be careful, since $K$ can be chosen so as to engineer `coincidences' between periods, such as $\int_1^{e^2}\frac{dx}x = 2$.
Therefore, Andr\'e insists one pay a price for making $K$ very large: the transcendence can come either from the periods over $K$, or from $K$ itself. In this way, one may consider the transcendence simultaneously of quantities like $\alpha$ and $\int_0^\alpha \omega$ by working over the field $\Q(\alpha)$. 

The general statement is as follows:

\begin{conj}[Andr\'e--Grothendieck period conjecture]
Let $M$ a Nori motive over $K$.  Then
\[\trdeg_\Q K(\mathrm{periods}\;\mathrm{of}\;M) \geq \dim\bfG_\mot(M).\]
\end{conj}

\subsubsection{Functional Setup}\label{secintrofunction}

To formulate a geometric analogue, we must find a replacement for the extension $\Q\subset \C$, as well as the integration map. Namely, we think of $\Q$ as the base-field over which we consider our varieties, and $\C$ as the extension field over which the period integrals are a-priori defined. As such, we replace $\Q$ with a complex function field $k$, and $\C$ with an appropriate field of meromorphic germs $k^{\an}$. Essentially, $k$ is the function field of a complex variety $S$ and varieties over $k$ are families $X\ra S$, so for an algebraic relative de Rham cohomology class $\omega\in H_{DR}^p(X/S)$ and a local flat section $\gamma_s$ of the $p$-homology of the fibers, our integration map considers the fiberwise periods $\int_{\gamma_s}\omega_s$ and yields meromorphic (as opposed to rational or algebraic) functions on the base $S$.  Moreover, for an intermediate field $k\subset K\subset k^\an$, if $K$ is the function field of a complex variety $T$, then the embedding $k\subset K$ yields a rational map $T\to S$ while the embedding $K\subset k^\an$ yields a local analytic section $\tau$ of $T\to S$ with Zariski dense image.  We therefore interpret the $k^\an$-valued periods of $H^p(X)$ for $X\to T$ as the pullback via $\tau$ of the periods over $T$.

Replacing $\Q\subset\C$ with the extension $k\subset k^{\an}$, we may formulate and prove the analogue of the Andr\'e--Grothendieck period conjecture:

\begin{thm}[see Theorem \ref{Andreperiod}]\label{Andreperiodintro}
Let $k$ be the function field of a complex algebraic variety and $k\subset K\subset k^\an$ where $K/k$ is finitely generated.  Then for any Nori motive $M$ over $K$ we have 
\[\trdeg_kK(\mathrm{periods}_{k^\an}\;\mathrm{of}\; M)\geq \dim \bfG_{\mot}(M/\C).\]

\end{thm}

Here $\bfG_{\mot}(M/\C)$ denotes the relative motivic Galois group (see Definition \ref{relmot}).  The analog of the Grothendieck period conjecture---namely the case $K=k$---was proven by Nori (unpublished) and Ayoub \cite{ayoubkz}.  As shown by Nori and Ayoub, in the functional setting the motivic Galois group $\bfG_\mot(M/\C)$ has a natural interpretation as the Zariski closure of the topological monodromy group acting on the local system associated to the Betti realization of $M$  (see Theorem \ref{exactGalois}).\footnote{In fact, Nori and Ayoub have a slightly artificial setup from this perspective, where they work only with subfields of $\C$. We find complex function fields to be a more natural context, so in \S4 we show how to go from their setup to this one at the cost of taking a slightly more complicated fiber functor.  }

\begin{remark}
Nori and Ayoub in fact prove the functional analog of the Kontsevich--Zagier period conjecture which states that the formal period ring injects into $k^\an$ via evaluation.  This is equivalent to the $K=k$ case of Theorem \ref{Andreperiodintro} together with the irreducibility of the torsor of isomorphisms between the de Rham and Betti realization functors.  See section 4 for further discussion.  The full Kontsevich--Zagier conjecture does not generalize to the setting of Theorem \ref{Andreperiodintro} without further assumptions on $K$, as $K$ itself may contain some period functions.  See \cite{ayoubsummary} (and specifically Remark 15) for a nice summary.  
\end{remark}


\subsection{Ax--Schanuel conjecture}
\subsubsection{Motivation}  We will deduce Theorem \ref{Andreperiodintro} from a version of the Ax--Schanuel conjecture for the analytic comparison between the flat and algebraic coordinates of an admissible variation of mixed Hodge structures.  The relation to an Ax--Schanuel type theorem is not surprising as Andr\'e's conjecture implies the classical Schanuel conjecture---see Example \ref{egAS}. In fact, it will turn out that these two statements are formally equivalent, if one restricts to studying variations of mixed Hodge structures which come from geometry.

In the past decade there has been much progress in functional transcendence, beginning with interest in unlikely intersection problems and Shimura varieties, and specifically on the period maps for variations of (mixed) Hodge structures \cite{PTAS,MPT,BTAS,chiu,GaoKlingler}. This amounts to studying the transcendence of the coordinates of the Hodge filtration in an appropriate flag variety with respect to a flat trivialization, which are roughly speaking ratios of certain period functions\footnote{Indeed, for an elliptic curve the coordinate $\tau$ for the Hodge filtration is the ratio of the two periods.}.  The main advantage of Theorem \ref{main} below is that it directly applies to the period functions themselves.  See \S\ref{elliptic} for a concrete example.

\def\full{\mathrm{full}}

\subsubsection{Main result}Let $S$ be an algebraic variety and $\cV=(V_\Z,W_\bullet V,F^\bullet V)$ an admissible variation of graded-polarizable integral mixed Hodge structures on $S$ (see for example \cite{PSmix} for background), where $V_\Z$ is an integral local system on $S^\an$, $W_\bullet V$ a (descending) filtration of $V_\Q:=V_\Z\otimes_{\Z_{S^\an}}\Q_{S^\an}$, and $F^\bullet V$ an (ascending) filtration of  $V_{\O^\an}:=V_\Z\otimes_{\Z_{S^\an}}\O_{S^\an}$.  Choosing a basepoint $s_0\in S$, we introduce the following notation:
\begin{itemize}
\item $V_{\Z,0}$ is the fiber of $V_\Z$ at $s_0$, and likewise for $V_{\Q,0},V_{\C,0}$,
    \item $(V_{\O},\nabla)$ is the canonical algebraic structure \cite{deligneext} on the flat vector bundle $(V_{\O^\an},\nabla)$, 
    \item $\VV:=\mathbb{A}(V_\O)$ its geometric total space with projection $\pi:\VV\to S$,
        \item $\G_\full\subset \bfGL(V_{\Q,0})$ is the full algebraic monodromy group, namely the $\Q$-Zariski closure of the image $\Gamma$ of $\pi_1(S^\an,s_0)\to \End(V_{\Q,0})$.
    \item $\G\subset \bfGL(V_{\Q,0})$ is the algebraic monodromy group, namely the identity component of $\G_\full$.
\end{itemize}
An irreducible subvariety $Z\subset S$ is contained in a proper \emph{weak Mumford-Tate subvariety} if and only if the algebraic monodromy group of the restriction $\cV_Z$ is smaller than $\G$, see \S\ref{corMT}.

\def\Isom{\mathbb{I}}
\def\calHom{\mathcal{H}\hspace{-.2em}\operatorname{om}}

Now, let $\cV_0$ be the trivial variation whose fiber is the fiber of $\cV$ over $s_0$.  Consider the variation $\cE:=\calHom(\cV,\cV_0)$, its underlying algebraic flat vector bundle $E_\cO$ with total space $\bE$, and let $\Isom\subset \bE$ be the open set of isomorphisms of the fibers (as vector spaces) in the geometric total space, which is naturally a $\bfGL(V_{\C,0})$-torsor over $S$ by post-composition.  We let $\widetilde{S^{\an}}$ be the minimal covering space of $S^{\an}$ which trivializes the local system $V_\Z$. Then solving the connection naturally gives a flat section $\widetilde{\sigma}_\cV:\widetilde{S^\an}\ra \widetilde{\Isom^\an}$ by sending a path to its flat transport operator. Projecting down we get a natural injective $S^\an$-map $\sigma_\cV:\widetilde{S^\an}\ra\Isom^\an$, whose image we denote by $\Sigma_\cV$.  We may also think of $\Sigma_\cV$ as the flat leaf of $\bE$ through the identity $\id:V_{\C,0}\to V_{\C,0}$ thought of as a point in the fiber above $s_0$.

Note that we may write the coordinates of this map as follows:  if we pick a basis $e_i$ for $V_{\Z,0}$ and a global meromorphic basis $\omega_j$ for $V_\cO$, then the coordinates for $\sigma_\cV$ are precisely the expansion of the $\omega_j$ in the flat continuation of the basis $e_i$. In the case that $V_\Z=R^nf_*\Z_X$, where $f:X\ra S$ is a smooth projective morphism, then the $\omega_j$ are relative de Rham cohomology classes and the coordinates of $\sigma_\cV$ are the period integrals of the $\omega_j$ over the dual homology basis to the $e_i$ along the fibers.

We shall show (see Lemma \ref{ZZar}) that that the Zariski closure $\Omega_\cV:=(\Sigma_\cV)^\Zar$ of $\Sigma_\cV$ is the $\G_\full(\C)$-orbit of $\Sigma_\cV$, and is therefore naturally a $\G_\full(\C)$-torsor which we call the \emph{period torsor}.  It has a natural flat connection restricted from $\mathbb{I}$ (see \S\ref{secttor}).  Moreover we have $\dim\Omega_\cV-\dim S=\dim\G $, which is the analog of Conjecture \ref{Grothperiod}.   

Our main theorem is:
\begin{thm}\label{main}
Suppose $W\subset \Omega_\cV$ is an algebraic subvariety and $U$ a component of $W\cap \Sigma_\cV$ such that 
\[\codim_{W} U<\dim\G.\label{ineqBDR}\]
Then the projection of $U$ to $S$ is contained in a weak Mumford--Tate subvariety.
\end{thm}

Given the setup of the discussion in \S\ref{secintrofunction}, this theorem immediately implies Theorem \ref{Andreperiodintro} by taking $W$ to be the Zariski closure of the image of the composition $\sigma_\cV\circ \tau$.  In fact, Theorem \ref{Andreperiodintro} implies Theorem \ref{main}, at least for variations coming from geometry (see Remark \ref{rmkequal}).  We give two proofs of Theorem \ref{main}, first as an application of the Ax--Schanuel theorem for principal bundles of \cite{diffas} and second using o-minimality, generalizing and using the results of \cite{PTAS,MPT,BTAS,chiu,GaoKlingler}.  The above theorem easily recovers previous Hodge-theoretic Ax--Schanuel theorems (see \S\ref{secotherAS}).

\subsection{Outline}
In \S\ref{secbackground} we collect some background needed for the proof of Theorem \ref{main}, and in \S\ref{secproof} we give both proofs of our main result. 

In \S\ref{secNori} we give a straightforward generalization of Nori's construction of motives over subfields $k\subset\C$ to complex function fields and establish the necessary ingredients to deduce Theorem \ref{Andreperiodintro} from Theorem \ref{main}. Since there seems to be a gap in the literature for Nori motives over function fields, we take this opportunity to also write down Nori's proof of the Grothendieck period conjecture in this setting. We do this by gathering theorems already present in the literature, mostly from \cite{huberms}.
We also point out that one can quickly deduce the full Kontsevich-Zagier conjecture (a theorem of Ayoub \cite{ayoubkz}) from the Grothendieck period conjecture, combined with our analytic description of the period torsor.

In \S\ref{elliptic} we discuss as an example an application to families of elliptic curves, and show how to use Theorem \ref{main} to formulate some related statements. We also prove the Ax--Lindemann conjecture for abelian differentials (recently conjectured by Klingler--Lerer \cite{klinglerlerer}).  Finally, we explain how our main theorems implies all previously known Hodge--theoretic Ax--Schanuel theorem. 

\subsection{Acknowledgements}

The authors wish to thank Pietro Corvaja and Umberto Zannier for asking the question that motivated this work. They are also greatly indebted to Jonathan Pila for suggesting the application to the Grothendieck period conjecture.  B.B. was partially
supported by NSF grant DMS-1848049.

\section{Background Results}\label{secbackground}
In this section, we briefly recall the statements from o-minimal geometry and Hodge theory that we will need.  We also prove some preliminary results that will be used in the proof of Theorem \ref{main}.

\subsection{o-minimality}\label{funset}

We shall be working throughout in the o-minimal structure $\R_{\an,\exp}$, see \cite{Dries} for background. We shall use the following definable Chow theorem of Peterzil--Starchenko:

\begin{thm}[Peterzil--Starchenko {\cite[Thm 4.5]{definechow}}]\label{definechow}
Let $Y$ be a quasiprojective algebraic variety, and let $A\subset Y$ be definable, complex analytic, and closed in $Y$. Then $A$ is algebraic.
\end{thm}

For an algebraic variety $S$ with a local system $V_\C$ (on $S^\an$), the total space $\mathbb{V}$ has a natural definable structure coming from its canonical algebraic structure.  On the other hand, another definable structure on $\mathbb{V}^\an$ is obtained by taking a definable cover of $S^\an$ by simply connected open sets and using flat coordinates.  In our case the two are the same by the following:
\begin{thm}[Bakker--Mullane {\cite[Theorem 1.2]{BMull}}]\label{alg=flat}
Let $V_\C$ be a local system underlying an admissible variation of graded-polarizable integral mixed Hodge structures on $S$.  Then the flat and algebraic definable structures on the total space $\mathbb{V}^\an$ are equivalent.
\end{thm}

We shall use the following precise corollary to provide a definable fundamental domain for $\Sigma_\cV$.

\begin{cor}\label{definableperiods}For $V_\C$ as in the theorem, let $f:B\ra S^\an$ be a definable map from a definable analytic space and $\sigma$ a flat section of $f^*V_\C$.  Then the corresponding lift $\bar f: B\to \mathbb{V}^\an$ is definable.
\end{cor}

We fix $\cF\subset \widetilde{S^{\an}}$ to be an open definable analytic subspace with simply connected components with a surjective map $\phi:\cF\ra S^\an$. Then by Proposition \ref{definableperiods} we have that $\Sigma_{\phi^*\cV}\subset \phi^*\Isom$ is definable, and therefore so is its image $\cF_\cV$ in $\Sigma_\cV$. Observe that $\cF_\cV$ surjects onto $S^\an$ and thus
$\Gamma\cdot\cF_\cV = \Sigma_\cV$.

\subsection{Point counting on transcendental sets}

Recall that the height of a rational point $\frac ab$ with $\gcd(a,b)=1$ is $H(\frac ab):=\max(|a|,|b|)$. For a point $q=(q_1,\dots,q_n)\in\Q^n$ we define
$H(q):=\max_i H(q_i)$. Finally, for any subset $X\subset\R^n$ we define $$N(X,T):=\#\{q\in X\cap \Q^n\mid H(q)\leq T\}.$$

Given a set $X\subset \R^n$, we set $X^{\alg}$ to be the union of all connected, positive dimensional, semi-algebraic subsets of $X$. Then we have the following theorem of Pila--Wilkie:

\begin{thm}[Pila--Wilkie {\cite[Thm 1.8]{PW}}]\label{pointcount}
For a set $X\subset\R^n$ definable in an o-minimal structure, and any $\epsilon>0$, we have $$N(X-X^{\alg},T) =  T^{o(1)}.$$
\end{thm}

In fact, we shall need the slightly stronger version:

\begin{thm}[Pila--Wilkie {\cite[Thm 3.6]{P}}]\label{pointcount2}
For a set $X\subset\R^n$ definable in an o-minimal structure, and any $\epsilon>0$, there is a definable family $W\subset X\times Y$ with semialgebraic fibers $W_y$, such that for any positive real number $T$, the rational points in $X$ of height at most $T$ are contained in the union of $T^{o(1)}$ of the fibers $W_y$.
\end{thm}

\subsection{Mumford--Tate groups}\label{MTback}

We follow the conventions of \cite{Andre}.  For a rational mixed Hodge structure $V=(V_\Q,W_\bullet V,F^\bullet V)$, recall that the group of weight zero Hodge classes of $V$ is $W_0V\cap F^0V$.  We define the Mumford--Tate group of $V$ to be the subgroup of $\GL(V_\Q)$ stabilizing each weight zero Hodge tensor in all tensor powers $V^{\otimes m}\otimes(V^{\dual})^{\otimes n}$. 

\begin{thm}[Andr\'e]\label{MTfacts}
For an admissible variation of graded-polarizable integral mixed Hodge structures $(V_\Z,W_\bullet V,F^\bullet V)$ on a smooth algebraic variety $S$, we denote by $\G_s$ the the Mumford--Tate group of the fiber $(V_{\Q,s},W_\bullet V_{\Q,s},F^\bullet V_{\C,s})$ at $s$ and $\bfH_s$ the connected component of the Zariski closure of the image of $\pi_1(S^\an,s)$ in $\GL(V_{\Q,s})$. Then
\begin{enumerate}
    \item For a generic $s\in S$ we have $\H_s\subset \G_s$;
    \item $\G_s$ is locally constant outside of a meager set;
    \item For a generic $s\in S$ the monodromy group $\H_s$ is a normal subgroup of the derived subgroup of $\G_s$.
    
\end{enumerate}
\end{thm}

\begin{proof}

These are \cite[Lemma 4, Thm 1]{Andre}.

\end{proof}

\begin{lemma}\label{ZZar}In the notation of the introduction,
$\Omega_\cV=\G(\C)\Sigma_\cV$ and $\G_\full(\C)$ acts simply transitively on fibers of $\Omega_\cV$ over $S$.

\end{lemma}
\begin{proof}
As $\Sigma_\cV$ is $\G(\Z)$-invariant, it follows that $(\Sigma_\cV)^{\Zar}$ is $\G(\C)$-invariant, and therefore contains $\G(\C)\Sigma_\cV$.  On the other hand, it follows from Lemma \ref{definableperiods} and locally choosing a flat section that $\G(\C)\Sigma_\cV$ is definable, and it is evidently a closed analytic subvariety. By Theorem \ref{definechow} it follows that it is algebraic and thus 
$\Omega_\cV=\G(\C)\Sigma_\cV$.

For the second part of the claim, note first that $\Sigma_{\cV}$ and hence $\Omega_{\cV}$ are invariant by the image of monodromy, and thus $\Omega_{\cV}$ is invariant under $\G_\full(\C)$.  Now, locally on $S^{\an}$  we have $\Sigma_\cV$ is given by a union of sections $s_i$ any two of which differ by an element of $\G_\full(\Z)$, and thus analytically locally on the base $\Omega^{\an}_{\cV} = \G_{\full}(\C)s_1$.
\end{proof}

\subsection{Flat torsors}\label{secttor}

Let $\G$ be a complex algebraic group.  Recall that an algebraic $\G$-torsor over a smooth variety $S$ is an algebraic $S$-variety $\pi:P\to S$ equipped with an algebraic left action by $\G$ such that the induced map
\[\G\times_S P\to P\times_S P,\;\;\; (g,x)\mapsto (gx,x)\]
is an isomorphism.  This means $\G(\C)$ acts simply transitively on the fibers of $P\to S$.

An algebraic flat connection on $P$ is an algebraic splitting of the extension
\[0\to T_{P/S}\to T_P\to \pi^*T_S\to 0\]
which is $\G$-invariant and with the property that the induced map $\pi^* T_S\to T_P$ is a foliation.  The leaves of $P$ are the leaves $\mathcal{L}$ of this foliation.

Given an algebraic flat $\G$-torsor $\pi:P\to S$, choosing a point $p_0\in P$ and setting $s_0=\pi(p_0)$ we obtain a monodromy representation $\rho:\pi_1(S,s_0)\to \G(\C)$ by solving the connection.  This data determines $\pi: P\to S$ analytically, as we recover
\[P^\an=(\widetilde{S^\an}\times \G(\C))/\Gamma.\]
where $\gamma\in\Gamma$ acts via the canonical right action on the first factor and via right multiplication by $\rho(\gamma)$ on the second.  Moreover, the constant sections $\widetilde{S^\an}\times g$ map to the leaves of $P$.  We define the full algebraic monodromy group (resp. algebraic monodromy group) of $P$ to be the Zariski closure (resp. identity component of the Zariski closure) of the image of $\rho$ in $\G$.

\begin{lemma}\label{everythingisaGtorsor}
$\Omega_\V$ is naturally an algebraic flat $\G_\full$-torsor for which $\Sigma_\V\subset\Omega_\V$ is a leaf and with algebraic monodromy $\G$.
\end{lemma}
\begin{proof}
The geometric vector bundle $\mathbb{E}$ has a natural flat connection, which restricts to a flat connection on $ \mathbb{I}\subset \mathbb{E}$ giving $\mathbb{I}$ the structure of an algebraic flat $\bfGL(V_{\C,0})$-torsor, with $\bfGL(V_{\C,0})$ acting by post-composition.  By Lemma \ref{ZZar} $\Omega_\V\subset \mathbb{I}$ is a union of leaves of $\mathbb{I}$, so the connection restricts to a flat connection on $\Omega_\V$.  The action of $\G_\full(\C)$ is simply transitive on fibers (again by Lemma \ref{ZZar}), and $\Sigma_\V$ is a leaf by definition.
\end{proof}

\subsection{Period domains and period maps}\label{perioddom}
We recall some definitions regarding weak Mumford--Tate domains.  See \cite{GGK,bbkt} for details.

\def\MT{\mathbf{MT}}

Let $\D_0$ be a period domain of graded-polarized integral mixed Hodge structures with generic Mumford--Tate group $\G_0$.  For any point $p\in\D$, let $\G$ be a normal $\Q$-subgroup of its Mumford--Tate group $\MT_p$ and $\bfU$ the unipotent radical of $\G$.  The orbit $\D:=\G(\R)\bfU(\C)\cdot p$ is a closed complex subspace of $\D_0$ whose generic Mumford--Tate group is $\MT_p$.  We call such a subspace a \emph{weak Mumford-Tate (sub)domain} of $\D_0$.  Each such $\D$ is naturally contained as a semialgebraic subset in a complex algebraic variety $\check{\D}$ called its dual.  

The quotient $\G(\Z)\backslash \D$ has the natural structure of a definable analytic variety, and for any period map $\tilde \phi:S^\an\to \G_0(\Z)\backslash\D_0$ the inverse image of $\G(\Z)\backslash\D$ is an algebraic subvariety of $S$ \cite{bbkt}.  We call each component a \emph{weak Mumford--Tate subvariety} of $S$.  Note that by Theorem \ref{MTfacts} and the above definitions, we have the following important property:
\begin{cor}\label{corMT}
Let $\cV$ be an admissible variation of graded-polarizable integral mixed Hodge structures on $S$.  Then a subvariety $S'\subset S$ is contained in a proper weak Mumford--Tate subvariety if and only if the algebraic monodromy group $\G'$ of the restriction $\cV_{S'}$ is strictly smaller then the algebraic monodromy group $\G$ of $\cV$.
\end{cor}

Suppose $\cV=(V_\Z,W_\bullet V_\Q,F^\bullet V)$ is an admissible variation of graded-polarizable integral mixed Hodge structures, and let $\phi:\widetilde {S^\an}\to \D$ be the associated period map, where $\pi:\widetilde{ S^\an}\to S$ is the minimal cover trivializing $V_\Z$.  Consider the map $\pi\times\phi:\widetilde {S^\an}\to S^\an\times\D$, which is a closed embedding of a component of $S^n\times_{\G(\Z)\backslash\D}\D$.  For the next section, we observe that $\pi\times \phi|_{\cF}$ is definable by \cite[Theorem 1.1]{bbkt}, where $\cF$ is as in \S\ref{funset}.



\def\bfM{\mathbf{M}}

\section{Two proofs of Theorem  \ref{main}}\label{secproof}

\subsection{First proof}
We first give a proof using the Ax--Schanuel for principal bundles of \cite{diffas}. For this we shall need the following definition:

\begin{defn}A complex algebraic group $\G$ is \emph{sparse} if every proper complex analytic Lie subgroup is contained in a proper complex algebraic Lie subgroup.  
\end{defn}

\begin{lemma}[{\cite[Lemma 3.3]{chiu2}}]\label{lemmachiu}  The algebraic monodromy of an admissible variation of polarizable integral mixed Hodge structures is sparse.

\end{lemma}

\begin{thm}[{\cite[Thm. A]{diffas}}]\label{torsorAS}  Suppose $\G$ is sparse.
Let $\pi:P\to S$ be an algebraic flat $\G$-torsor, $W\subset P$ a subvariety, $\mathcal{L}$ a leaf of $P$, and $U$ a component of $W^\an\cap\mathcal{L}$.  If
\[\codim_W U<\dim \G\]
then $\pi(U)^\Zar\subset S$ has algebraic monodromy of strictly smaller dimension than $\G$.
\end{thm}

\begin{proof}[Proof of Theorem \ref{main}] Note by Lemma \ref{everythingisaGtorsor} that $\Omega_\cV$ is an algebraic flat $\G_{\full}$-torsor and that $\Sigma_\cV$ is a leaf. Thus, the theorem follows immediately from Theorem \ref{torsorAS} and Lemma \ref{lemmachiu}. 

\end{proof}
\begin{remark}
The proof only uses Hodge theory to establish that the algebraic monodromy of the underlying local system is sparse, and it is natural to ask whether Theorem \ref{torsorAS} is true for any local system. The following example shows it is not.
\end{remark}
\begin{eg}
Let $S=A$ be a simple abelian surface and let $\omega$ be a nonzero differential 1-form.  Consider the local system with monodromy
\[\pi_1(A,0)\to\G_m^2,\;\;\;\gamma\mapsto \left(e^{\int_\gamma\omega},e^{\lambda\int_\gamma\omega}\right).\]
The associated $\G_m^2$-torsor $\pi:P \to A$ (equipped with its canonical algebraic structure) has a section $s:A\to P$ whose lift is given by
\[\tilde s:\widetilde {A^\an}\mapsto \G_m^2,\;\;\; a\mapsto \left(e^{\int_0^a\omega},e^{\lambda\int_0^a\omega}\right)\]
which is algebraic by GAGA.  For $\lambda$ irrational, the fibers of $\tilde s$ are one-dimensional and project to the intersections $\cL\cap s(A)$, which are therefore one-dimensional as well.  The algebraic monodromy is all of $\G_m^2$, but if a fiber were not Zariski dense in $A$ it would necessarily be an elliptic curve factor of $A$. 
\end{eg}

\subsection{Second proof}
In this section we describe the proof of Theorem \ref{main} using o-minimality in the spirit of \cite{MPT,BTAS,chiu,GaoKlingler}.

With the notation as in the setup of Theorem \ref{main}, we prove the conclusion by induction on the triple $(\dim S,\dim W-\dim U,-\dim U)$ with the lexicographical ordering, the base case of $\dim S=0$ being trivial. We thus assume that the theorem is valid for all lexicographically previous triples.  Suppose we have a $W\subset \Omega_\cV$ as in the statement of the theorem and a component $U$ of $W\cap\Sigma_\cV$ whose projection $\pi(U)$ is not contained in a proper weak Mumford--Tate subvariety of $S$.  Note that $\pi(U)$ is Zariski dense in $S$ by the inductive hypothesis, as otherwise we could replace $S$ with the Zariski closure $S'$ of $\pi(U)$ and $W$ with the intersection $W\cap \Omega_{\cV_{S'}}$.




Recall that $\G(\C)$ acts on $\Omega_\cV$ (algebraically) by post-composition.  Let $\Gamma\subset\G(\Q)$ be the image of the monodromy representation of $V_\Z$ (possibly after replacing $S$ with a finite cover).  As in \cite[\S3]{MPT}, we consider the component $M$ of the Hilbert scheme $\Hilb\left(\overline{\Omega}_\cV\right)$ containing the closure of $W$ for some equivariant algebraic compactification $\overline{\Omega}_\cV$ of $\Omega_\cV$.  Let $\cW\subset \Omega_\cV\times M$ be the universal family and $\cW_{\Sigma}$ the intersection with $\Sigma_\cV\times M^\an$, which is $\Gamma$-invariant and proper over $\Sigma_\cV$.  The quotient $\cU:=\Gamma\backslash \cW_{\Sigma}$ is naturally a definable analytic variety as follows.  Letting $\cF_\cV\subset\Sigma_\cV$ be the subset from \S\ref{funset}, we set $\cG_\cV= \cF_\cV\times M^\an$, which is an open definable fundamental set for $\Gamma$ on $\Sigma_V\times M^\an$.  Let $\sim$ be the induced definable \'etale equivalence relation on $\cG_\cV$, and we take the definable structure induced by the identification $\cU\cong(\cW_{\Sigma}\cap \cG_\cV)/\sim$. 

The natural map $\cU\to S^\an$ is proper definable analytic.  We think of $\cU$ as parametrizing $\Gamma$-orbits of pairs $(W',p)$ with $[W']\in M$ and $p\in W'^\an\cap\Sigma_\cV$. There is a closed $\Gamma$-invariant definable analytic subvariety $A\subset \cU$ parametrizing pairs $(W',p)$ with $\dim_p(W'^\an\cap \Sigma_\cV)\geq\dim U$.  If $A_0$ is an irreducible component of $A$ containing $(W,p)$ for all $p\in U$, then $A_0$ descends to a closed definable analytic subvariety $B_0\subset \cU$ by taking the quotient $B_0:=(A_0\cap \cG_\cV)/\sim$.  

We have a natural proper definable analytic map $q:B_0\to S^\an$. By Theorem \ref{definechow} the image is algebraic and therefore $q$ is surjective by the inductive hypothesis and the properness of the map. Note that the monodromy $\Gamma_0$ of the pullback $q^*V_\Z$ stabilizes $A_0$.  Since $B_0$ surjects onto $S$, we have that $\Gamma_0\subset \Gamma$ is finite index, so the identity component of the $\Q$-Zariski closure of $\Gamma_0$ is $\G$. 

Letting $Y\subset M^\an$ be the projection of $A_0$, we let $H_{\mathrm{gen}}$ be the stabilizer of a very general point of $Y$, and $\H$ the identity component of its $\Q$-Zariski closure.  Note that every point of $Y$---in particular $[W]$---is stabilized by $\H(\C)$.  Moreover, since $\Gamma_0$ sends a very general point to a very general point, it follows that $H_{\mathrm{gen}}$ is normalized by $\Gamma_0$, hence $\H\subset \G$ is normal.
\begin{claim}
$\H=\G$.
\end{claim}

\begin{proof}

Let $Z=S^\an\times_{\G(\Z)\backslash\cD}\cD\subset S^\an\times\check\cD^\an$ where $\mathcal{D}$ is the relevant weak Mumford--Tate domain and $\check{\mathcal{D}}$ its dual.  Let $F\subset Z$ be a definable fundamental set for the action of $\G(\Z)$.  We have the following result of Chiu \cite{chiu}:
\begin{prop}\label{chiuprop}
Let $\H\subset\G$ be a normal $\Q$-subgroup and $K\subset Z$ a closed irreducible complex analytic subvariety which is stabilized by $\H(\Z)$ and for which $K\cap\gamma F$ is definable for any $\gamma\in \G(\Z)$.  Let 
\[J=\{\gamma\in\G(\Z)\mid K\cap \gamma F\neq\varnothing\}\]
and note that $\H(\Z)$ acts on $J$.  Then either $\H(\Z)\backslash J$ is finite or has polynomially many integer points.
\end{prop}

\begin{proof}
This is proven in \cite{chiu}. The case that $\H(\Z)\backslash J$ has polynomially many points corresponds to cases (1) and (2) in the trichotomy at the end of section 7 in \cite{chiu}, and is proven in sections 8 and 9 respectively. The case that $\H(\Z)\backslash J$ is finite is case (3).

Chiu proves it for the specific set $U$ in his notation, but the proof works verbatim for an arbitrary $K$ as in the statement of Proposition \ref{chiuprop}.

\end{proof}
We apply this proposition to $\Sigma_\cV$, which is identified with a component of $Z$ as in \S\ref{perioddom}, using $K=W^\an\cap \Sigma_\cV$ and $F=\cF_\cV$.  Consider  
\[I=\{g\in\G(\R)\mid \dim(g^{-1}W^\an\cap  \cF_\cV)=\dim U\}.\] Note that the $J$ in Proposition \ref{chiuprop} is a subset of $I(\Z)$. Moreover, $I$ is definable. Thus, by Theorem \ref{pointcount}, either

\begin{enumerate}

    \item $I$ contains a semi-algebraic curve $C$ with non-constant image in $\H(\R)\backslash I$, or
     \item $\H(\Z)\backslash J$ is finite.
\end{enumerate}

Suppose first that we are in case (1). Then $I$ contains a semialgebraic curve $C$. We claim that $W$ is not stabilized by $C$. If it were, then it would be stabilized by $C\cdot C^{-1}$, which contains an integer point not in $\H(\Z)$ (in fact we can arrange it to contain arbitrarily many by the conclusion of the stronger \ref{pointcount2}). This is a contradiction by the definition of $\H$. 

Therefore $c^{-1}W$ varies with $c\in C$. If $c^{-1}U$ does not vary with $c\in C$ then we may replace $W$ with $W\cap c^{-1}W$ for a generic element $c$ and obtain a lexicographically smaller counterexample. Else, if $c^{-1}U$ does vary with $c\in C$, we may replace $W$ with $C^{\Zar}W$ and obtain a lexicographically smaller counterexample.

Thus, we may assume we are in case (2). In this case, it follows that the image $\pi(U)$ of $U$ in $S$ is definable, hence algebraic by Theorem \ref{definechow}. Since the monodromy of $\pi(U)$ and $U$ are the same, if we have $\H\neq\G$ then $\pi(U)$ would be contained in a proper weakly special subvariety, which contradicts the assumption on $U$. 
\end{proof}

We may therefore suppose $W$ is $\G$-invariant, but then it is obvious that $\codim_WU\geq \dim \G$, and this contradiction proves the theorem.


\section{Nori motives over complex functions fields}\label{secNori}

\subsection{Outline}

In this section we prove Theorem \ref{Andreperiodintro}, whose statement is formulated in terms of functions fields of complex algebraic varieties.  Nori's category of motives is defined only for subfields of $\C$, and due to its reliance on a Betti realization functor the generalization of his construction to complex function fields requires a little care.  In \S\ref{noriclass} we recall the original construction of Nori, in \S\ref{conjclass} we precisely state the classical Grothendieck period conjecture and its generalization due to Andr\'e.  In \S\ref{norinew} we make the necessary modifications to Nori's construction and in \S\ref{noriGalois} we relate the relative motivic Galois group to the algebraic monodromy.  In \S\ref{noritorsor} we relate the torsor of comparisons between Betti and de Rham realizations to our period torsor $\Omega_\V$.  In \S\ref{noriKZ} we show has this perspective gives a simple perspective on the geometric Kontsevich--Zagier conjecture.  Finally in \S\ref{conjnew} we formulate and prove the precise version of Theorem \ref{Andreperiodintro}.  The reader who is willing to assume a reasonable category of motives over a complex function field together with the statement of Theorem \ref{pi1} can skip directly to the proof.

\subsection{Nori motives}\label{noriclass}
In this section we briefly recall Nori motives, which will provide for us a Tannakian category of motives in both the classical and functional setting, and therefore a motivic Galois group.  Ayoub \cite{ayoubkz} (see also \cite{ayoubgaloisii}) takes a slightly different approach, using Voevodsky's theory to define such a group directly, but they are canonically the same \cite{isomgalois}.   Essentially, for any subfield $k\subset\C$, the category of Nori $k$-motives will be the abelian subcategory of $\Q\mbox{-mod}$ generated by singular cohomology groups $H^i((X_\C)^\an,\Q)$ of $k$-varieties $X$ together with all morphisms that can be constructed naturally from maps of $k$-varieties.  The main reference is \cite{huberms}.

By a diagram $D$ we mean a directed graph\footnote{with possibly infinitely many vertices and edges.} with the obvious notion of morphism.  Note that for any category $\mathcal{C}$ there is a natural underlying diagram, and for every functor $\mathcal{C}\to\mathcal{C}'$ an underlying morphism of diagrams.  A representation $F:D\to\mathcal{C}$ of a diagram $D$ in a category $\mathcal{C}$ is a morphism on the level of diagrams.  Concretely, $F$ assigns an object of $\mathcal{C}$ to each vertex of $D$, and a morphism of $\mathcal{C}$ to each edge of $D$ with the obvious compatibility on the source and target.

Let $k$ be a field with an embedding $\iota:k\to \C$.  The diagram $\Pairs^\eff(k,\iota)$ has as vertices triples $(X,Y,i)$ with $X$ an algebraic variety over $k$, $Y\subset X$ a closed subvariety (defined over $k$), and $i\in \Z$.  The edges of $\Pairs^\eff$ consist of
\begin{itemize}
    \item for each morphism $f:X\to X'$ with $f(Y)\subset Y'$ and integer $i$ there is an edge $f^*:(X',Y',i)\to (X,Y,i)$;
    \item for each chain $X\supset Y\supset Z$ of varieties (the inclusions being of closed subvarieties) and each integer $i$ there is an edge $\partial:(Y,Z,i)\to(X,Y,i+1)$. 
\end{itemize}
Betti cohomology $(X,Y,i)\mapsto H^i((X_\C)^\an,(Y_\C)^\an,\Q)$ defines a natural representation 
\[\mathrm{Betti}_\iota:\Pairs^\eff(k)\to \Q\mbox{-mod}\]
where the edge $f^*$ is sent to the pullback via $f$ and $\partial$ is sent to the coboundary map in the long exact sequence of the triple. 
\begin{thm}[Nori, see {\cite[Theorems 7.1.13 and 9.1.5]{huberms}}]\label{norimot}Let $k$ be a field with an embedding $\iota:k\to \C$.
\begin{enumerate}
    \item There is a $\Q$-linear abelian category $\mathcal{MM}^\eff_\mathrm{Nori}(k,\iota)$
    together with a representation $\mathcal{H}_\iota:\Pairs^\eff(k)\to \mathcal{MM}^\eff_\mathrm{Nori}(k,\iota)$ and a faithful exact $\Q$-linear functor $H_\iota:\mathcal{MM}^\eff_\mathrm{Nori}(k,\iota)\to \Q\text{-}\mathrm{mod}$ which is uniquely determined by the property that given
    \begin{itemize}
        \item a $\Q$-linear abelian category $\mathcal{A}$
        \item a representation $\mathcal{F}:\Pairs^\eff(k)\to\mathcal{A}$
        \item a faithful exact $\Q$-linear functor $F:\mathcal{A}\to \Q\mbox{-mod}$ such that the solid part of the diagram below commutes (on the level of diagrams)
        \[\begin{tikzcd}
        &\mathcal{MM}^\eff_\mathrm{Nori}(k,\iota)\ar[rd,"H_\iota"]\ar[dd,dashed,"\Phi"]&\\
\Pairs^\eff(k)\ar[ru,"\mathcal{H}_\iota"]\ar[rd,swap,"\mathcal{F}"]&&\Q\mbox{-mod}\\
&\mathcal{A}\ar[ru,swap,"F"]&
        \end{tikzcd}\]
        \end{itemize}
        there exists a unique faithful exact $\Q$-linear functor $\Phi:\mathcal{MM}_{\mathrm{Nori}}^\eff(k,\iota)\to\mathcal{A}$  making the diagram commute (on the level of diagrams).
    \item The category $\mathcal{MM}_\mathrm{Nori}^\eff(k,\iota)$ has a natural commutative tensor product with unit such that $H_B$ is a tensor functor.
        \item The category $\mathcal{MM}_\mathrm{Nori}(k,\iota)$ obtained from $\mathcal{MM}_\mathrm{Nori}^\eff(k,\iota)$ by inverting $\mathcal{H}^1_\iota(\G_m,\{1\})$ is a rigid tensor category with fiber functor $H_\iota$.  Here we denote $\mathcal{H}_\iota^i(X,Y):=\mathcal{H}_\iota(X,Y,i)$.
        \item $\mathcal{MM}_\mathrm{Nori}(k,\iota)$ with $H_\iota$ as its fiber functor is naturally equivalent to the category of representations of a pro-algebraic $\Q$-group $\bfG_\mathrm{mot}(k,\iota)$ with its natural fiber functor.
\end{enumerate}
\end{thm}

Note in particular that we have $H_\iota\circ\mathcal{H}_\iota=\mathrm{Betti}_\iota$.  We refer to $\NoriMot^\eff(k,\iota)$ (resp $\NoriMot(k,\iota)$) as the category of effective Nori $(k,\iota)$-motives (resp. Nori $(k,\iota)$-motives).  

\begin{defn}
For a Nori $(k,\iota)$-motive $M$ we define $\bfG_\mot(M,\iota)$ to be the image of the natural map $\bfG_\mot(k,\iota)\to\bfGL(H(M))$.  It is an algebraic $\Q$-group. 
\end{defn}

\subsection{Classical period conjectures}\label{conjclass}It will be useful to have a category of pairs of vector spaces equipped with a comparison over a fixed field extension.  Let $\iota:k\to L$ be an embedding of characteristic 0 fields.  Define $(k,\Q)_\iota\mbox{-mod}$ to be the category of triples $(U,V,\phi)$ of a $k$-vector space $U$ a $\Q$-vector space $V$, and an isomorphism $\phi:U\otimes_k L\to V\otimes_\Q L$ with the obvious notion of morphism.

For a field $k$ and an embedding $\iota:k\to\C$ we have a natural representation
\[\Pairs^\eff(k)\to (k,\Q)_\iota\mbox{-mod}\]
given by sending $(X,Y,i)$ to $(H_{DR}^i(X,Y),H^i(X,Y,\Q),\phi_{X,Y,i})$ where $\phi_{X,Y,i}$ is the natural comparison given by integration.  By the universal property this extends to a functor
\[\NoriMot^\eff(k,\iota)\to (k,\Q)_\iota\mbox{-mod},\;\;\;\;M\mapsto (H_{DR}(M),H_\iota(M),\phi_M)\]
which in turn extends to a functor 
\[\NoriMot(k,\iota)\to (k,\Q)_\iota\mbox{-mod},\;\;\;\;M\mapsto (H_{DR}(M),H_\iota(M),\phi_M).\]
\begin{defn}
For $\iota:k\to \C$ a field embedding and $M$ a Nori $(k,\iota)$-motive we define 
\[k(\mathrm{periods}_\iota\;\mathrm{of}\;M)\subset\C\]
to be the field of definition of the comparison $\phi_M$.
\end{defn}
Concretely $k(\mathrm{periods}_\iota\;\mathrm{of}\;M)$ is obtained by adjoining the periods of $k$-rational de Rham classes of $M$ to $k$.
\begin{conj}[Grothendieck period conjecture]
Let $\iota:k\to\C$ be an embedding of a number field and $M$ a Nori $(k,\iota)$-motive.  Then
\[\trdeg_\Q k(\mathrm{periods}_\iota\;\mathrm{of}\;M)=\dim\bfG_\mot(M,\iota).\]
\end{conj}

One downside of the Grothendieck period conjecture is that it does not imply the other major conjecture about transcendence: the Schanuel conjecture about exponentials. Andr\'e proposed the following strengthening to address this:

\begin{conj}[Andr\'e--Grothendieck period conjecture]  Let $\iota:k\to\C$ be any field embedding and $M$ a Nori $(k,\iota)$-motive.  Then
\[\trdeg_\Q k(\mathrm{periods}_\iota\;\mathrm{of}\;M)\geq\bfG_\mot(M,\iota).\]
\end{conj}
\begin{eg}\label{egAS}
Let $\alpha_1,\dots,\alpha_n\in\C^*$ be multiplicatively independent, and consider $X=\G_m$ and $ Y=\{1, \alpha_1,\dots,\alpha_n\}\subset\G_m$ over $k=\Q(\alpha_1,\dots,\alpha_n)$.  Then $H^1_{DR}(X,Y)$ is spanned by $\frac{dx}{x}$, and the differences $[1]^\vee-[\alpha_i]^\vee$, whereas $H_1(X,Y,\Q)$ is spanned by paths between $1$ and the different 
$\alpha_i$, and  the loop around $0$. Thus the integrals are all integers, as well as $2\pi i, \log\alpha_1,\ldots,\log\alpha_n$.
Hence in this case, Andr\'e's conjecture says that
$$\trdeg_\Q(2\pi i, \alpha_1,\ldots,\alpha_n,\log \alpha_1,\ldots, \log\alpha_n)\geq \dim \bfG_{\mot}(\mathcal{H}^1_\iota(X,Y)).$$

On the other hand, Andr\'e shows \cite{Andre1motives} that $\bfG_{\mot}(\mathcal{H}^1_\iota(X,Y))$ is the same as the Mumford--Tate group of the mixed Hodge structure $H^1(X,Y)$---and this is easily computed to be an extension of $\G_m$ by $n$ copies of $\G_a$. Thus Andr\'e's conjecture says that 
$$\trdeg_\Q(2\pi i, \alpha_1,\ldots,\alpha_n,\log \alpha_1,\dots, \log\alpha_n)\geq n+1 $$ and therefore that
$$\trdeg_\Q(\alpha_1,\ldots,\alpha_n,\log \alpha_1,\dots, \log\alpha_n)\geq n $$ 
which is precisely the statement of Schanuel's conjecture.

\end{eg}

\subsection{Nori motives in the functional setting}\label{norinew}  In the functional setting we would like to replace $k$ with the function field of a complex algebraic variety, in which case the Betti realization of a Nori motive $M$ defined over $k$ should be the Betti cohomology of a generic fiber of $M$ once we spread $M$ out over a model $S$ of $k$.  In this section we make these ideas precise.

Throughout we often take $\iota_0:k_0\to \C$ to be a field embedding and $k_0\subset k$ a finitely generated extension.  By analytification we always mean analytification as $k_0$-varieties, unless otherwise specified.

\begin{defn}
An \emph{arc point} $\gamma$ of a topological space $\cS$ is an equivalence class of  continuous paths $\gamma:(0,a)\to \cS$ for $a>0$, where we say two arc points $\gamma,\gamma'$ are equivalent if they agree on some interval $(0,\delta)$ for $\delta>0$. 
 
\end{defn}

Note that classical points may be thought of as arc points via the constant maps. 

\begin{lemma}\label{stablearc}

Let $S$ be a model of $k$ over $k_0$ and let $\gamma:(0,a)\ra S^{\an}$ be an arc such that 

\begin{enumerate}
    \item $\gamma$ extends real analytically over $(-\epsilon,a+\epsilon)$ for some $\epsilon>0$
    \item $\gamma $ is not contained in any proper $k_0$-algebraic subvariety of $S^{\an}$.
  \end{enumerate}  
Then $\gamma$ defines an arc point of every model $S'$ of $k$ compatibly with respect to morphisms of models.

\end{lemma}

\begin{proof}

Let $V\subset S$ be a (nonempty) open subscheme. Then by the two assumptions, the set $\gamma^{-1}(S^{\an}\backslash V^{\an})$ is finite, and therefore for sufficiently small $0<b<a$ the map $\gamma|_{ (0,b)}$ factors through $V$. Thus $\gamma$ defines an arc point of any (nonempty) open subscheme of $S$, and therefore of any other model, since any two models agree on an open set.

\end{proof}
For an arc point of $S^\an$ satisfying the conditions in the lemma, we say that the arc point induced on any other model is \emph{stable}.
\begin{defn}
Let $k_0\subset k$ be fields such that $k$ is finitely generated over $k_0$ and $k_0$ is algebraically closed in $k$. Let $\iota_0:k_0\ra \C$ be a field embedding. We say that $\gamma$ is a $(k_0,\iota_0)$-\emph{arc point} of $k$ if $\gamma$ is a compatible choice of stable arc point of $S^{\an}$ for any model $S$ of $k$ over $k_0$. Note that this is equivalent to a choice of stable arc point on one model by Lemma \ref{stablearc}. Note also that a complex embedding $\iota:k\ra\C$ extending $\iota_0$ naturally gives a $(k_0,\iota_0)$-arc point of $k$.    
\end{defn}

\begin{remark}\label{pusharc}
Suppose we have two triples $(k_0,\iota_0,k)$ and $(\ell_0,\lambda_0,\ell)$ as in the definition with a containment $(k_0,\iota_0,k)\subset (\ell_0,\lambda_0,\ell)$ in the obvious way.  Then any $(\ell_0,\lambda_0)$-arc point of $\ell$ naturally pushes forward to a $(k_0,\iota_0)$-arc point of $k$.
\end{remark}

\begin{defn}\hspace{1in}
\begin{enumerate}
    \item 
Let $\gamma_1,\gamma_2$ be two arc points of a topological space $\cS$.  A \emph{path} from $\gamma_1$ to $\gamma_2$ of $\cS$ is a continuous map $\phi:(0,1)\ra \cS$ such that $\phi$ is equivalent to $\gamma_1$, and $\phi\circ(1-x)$ is equivalent to $\gamma_2$.  A \emph{homotopy} of paths $\phi_1$ and $\phi_2$ from $\gamma_1$ to $\gamma_2$ is an ordinary homotopy $\phi_t$ between $\phi_1$ and $\phi_2$ such that each $\phi_t$ is a path from $\gamma_1$ to $\gamma_2$. 
\item Let $\gamma_1,\gamma_2$ be two $(k_0,\iota_0)$-arc points of $k$.  A \emph{homotopy class of paths} from $\gamma_1$ to $\gamma_2$ is a homotopy class of paths from $\gamma_1$ to $\gamma_2$ in $(k/k_0,\iota_0)^{\an}$.
\end{enumerate}
\end{defn}
In the next section we will need the following notion:
\begin{defn}
Let  $\cS$ be a topological space and $\gamma:(0,a)\ra \cS$ an arc point. For $0<\epsilon<a$ we may identify all the  $\pi_1(\cS,\gamma(\epsilon))$ using the path $\gamma$. We call the equivalence class of all these groups $\pi_1(\cS,\gamma)$. Note that this is (non-canonically) isomorphic to the usual topological fundamental group.  
\end{defn}

\begin{lemma}\label{arcpointconnect}
Suppose that $k/k_0$ is a finitely generated extension of fields such that $k_0$ is algebraically closed in $k$. Then there exists a homotopy class of paths between any two $(k_0,\iota_0)$-arc points of $k$.

\end{lemma}

\begin{proof}

Note that by the assumption, any model $S$ of $k$ is geometrically connected. Now let $U\subset V$ be connected open sets, and consider $U^{\an}\subset V^{\an}$. Since the real codimension is at least 2, it is follows that homotopy classes of paths in $U^{\an}$ surject onto homotopy classes of paths in $V^{\an}$. Therefore it is sufficient to prove that for any connected manifold $T$, any two arc points $\gamma_1,\gamma_2$ of $T$ have a path between them. But this is trivial since connected manifolds are path-connected.

\end{proof}

For $X$ a $k$-variety, $Y\subset X$ a closed $k$-subvariety, and $\gamma $ a $(k_0,\iota_0)$-arc point of $k$ we define 
\[H^i_\gamma(X,Y):=H^i(X^\an_{\gamma},Y^\an_\gamma,\Q).\]
Here we spread out $X,Y$ over a model $S$ of $k$, shrink $S$ so that $H^i(X^\an_s,Y^\an_s)$ forms a local system, and represent the arc point by $\gamma:(0,a)\to S^\an$.  Then $X^\an_\gamma$ is defined to be the inverse image of $\img\gamma$ in $X^\an$.  This naturally yields a representation
\[\mathrm{Betti}_\gamma:\Pairs^\eff(k)\to \Q\mbox{-mod},\;\;\;\; (X,Y,i)\mapsto H^i_\gamma(X,Y)\]
where the edge $f^*$ is sent to the pullback via $f$ and $\partial$ is sent to the coboundary map in the long exact sequence of the triple. 

\begin{lemma}\label{eqfiberfun}
Let $\phi$ be a homotopy class of paths between two $(k_0,\iota_0)$-arc points $\gamma_1,\gamma_2$ of $k$.  Then $\phi$ gives a natural equivalence $\mathrm{Betti}_{\gamma_1}\cong \mathrm{Betti}_{\gamma_2}$.
\end{lemma}
\begin{proof}

Given $X,Y$ over $k$ we may spread out to a model $S$ of $K$. Moreover, by shrinking $S$ we may assume that the fibers $H^*_{\gamma_i}(X,Y)$ form a local system over $S$.

Now $\phi$ gives a homotopy equivalence class of paths between $\gamma_1$ and $\gamma_2$ as arc points of $S$. Thus, we get a natural identification of $H^*_{\gamma_1}(X,Y)$ with  $H^*_{\gamma_2}(X,Y)$, as desired. 

\end{proof}

\begin{prop}\label{functionalprop}
Let $\gamma$ be a $(k_0,\iota_0)$-arc point of $k$. 
\begin{enumerate}[label=(\alph*)]
    \item There exist categories of Nori motives $\NoriMot^{\eff}(k,\gamma),\NoriMot(k,\gamma)$ with representation $\mathcal{H}_\gamma:\Pairs^\eff(k)\to\NoriMot^{\eff}(k,\gamma)$ and functors $H_\gamma:\NoriMot^{(\eff)}(k,\gamma)\to \Q\text{-}\mathrm{mod}$ satisfying all of the properties of Theorem \ref{norimot}.
    \item   For any finitely generated extension $k\subset K$ with compatible arc points $\gamma_k$ and $\gamma_K$ as in Remark \ref{pusharc}, there is a natural base-change functor $\NoriMot^{(\eff)}(K,\gamma_K)\to\NoriMot^{(\eff)}(k,\gamma_k)$ which respects the tensor product structure.
\end{enumerate}

\end{prop}
\begin{proof}

Consider first the case that $k$ (and therefore $k_0$) is countable. By Lemma \ref{arcpointconnect} any two arc points have a homotopy class of paths between them, so it is sufficient to consider a single point by Lemma \ref{eqfiberfun} and this is the case of Theorem \ref{norimot}.

Next consider the general case of part (a).  From the diagram category construction of \cite[\S 7]{huberms} we obtain from the representation $H_{\gamma_k}:\Pairs^\eff(k)\to\Q\mbox{-mod}$ a $\Q$-linear abelian category $\NoriMot^\eff(k,\gamma_k)$ with representation $\mathcal{H}_{\gamma_k}:\Pairs^\eff(k)\to \NoriMot^\eff(k,\gamma_k)$ and a faithful exact $\Q$-linear functor $H_{\gamma_k}:\NoriMot^\eff(k,\gamma_k)\to\Q\mbox{-mod}$ satisfying property (a) of Theorem \ref{norimot}. 

Consider the directed set $I$ of countable subfields $\ell\subset k$ and let $\ell_0=\ell\cap k_0$.  The arc point $\gamma_k$ induces an arc point $\gamma_\ell$ of $\ell$.  For an inclusion $\ell\subset\ell'$ the natural base-change morphism $\Pairs^\eff(\ell)\to\Pairs^\eff(\ell')$ of diagrams yields a base-change functor $\NoriMot^\eff(\ell,\gamma_\ell)\to\NoriMot^\eff(\ell',\gamma_{\ell'})$.  As every variety over $k$ is defined over some $\ell$, we naturally have 
\[\Pairs^\eff(k)=\underset{\ell\in I}{\colim}\,\Pairs^\eff(\ell)\]
as diagrams and moreover \[H_{\gamma_k}=\underset{\ell\in I}{\colim}\,H_{\gamma_\ell}.\]
as representations.  By the universal property we then have a canonical identification
\[\NoriMot^\eff(k,\gamma_k)=2\text{-}\underset{\ell\in I}{\colim}\,\NoriMot^\eff(\ell,\gamma_\ell).\]
Indeed, the diagram category is constructed as a 2-colimit over finite subdiagrams.  Properties (2) and (3) then follow.  Property (4) is by Tannakian duality, though in this case we can directly see that (4) holds with $\bfG_\mot(k,\gamma):=\lim_{\ell\in I}\bfG_\mot(\ell,\gamma_\ell)$.

Part (b) again follows from the corresponding statement in the countable case.
\end{proof}

\begin{defn}\label{relmot}In the situation of part (2) of the above proposition, the relative motivic Galois group $\bfG_\mot(K/k,\gamma_K)$ is the kernel of $\bfG_\mot(K,\gamma_K)\to\bfG_\mot(k,\gamma_k)$.  For any Nori $(K,\gamma_K)$-motive $M$, the relative motivic Galois group $\bfG_\mot(M/k,\gamma_K)$ is the image of the natural map $\bfG_\mot(K/k,\gamma_K)\to \bfG_\mot(M,\gamma_K)$.
\end{defn}


It follows naturally from the construction that paths between stable arc points give compatible isomorphisms between the $\pi_1$ groups, the functors $H_\gamma$, the categories of Nori motives, and the motivic Galois groups.

\def\LocSys{\mathrm{LocSys}}
\subsection{The relative motivic Galois group}\label{noriGalois}  In this section we relate the relative motivic Galois group (over a point) to the algebraic monodromy group.  For $k$ a subfield of $\C$ this is a result of Ayoub \cite{ayoubgaloisii} (in a different but equivalent context by \cite{isomgalois}, as mentioned above) and of Nori (unpublished).  The details of the latter argument have recently been worked out by Mostaed \cite{mostaed}.

As in the previous section, let $k_0\subset k$ be a finitely generated field extension such that $k_0$ is algebraically closed in $k$ and let $\iota_0:k_0\to \C$ be a field embedding.  Let $\gamma$ be a $(k_0,\iota_0)$-arc point of $k$.
\begin{defn}
We define $(k/\C)^{\an}$ to be the pro-manifold obtained by taking the system of manifolds $S^{\an}$ for (smooth) models $S$ of $k$.  Given an arc point $\gamma$ of $k$ we obtain an arc point of every model as in the previous section and we define a \emph{homotopy class of paths} of $(k/\C)^\an$ as a compatible system of homotopy classes of paths from $\gamma$ to $\gamma$.  The resulting fundamental group $\pi_1((k/\C)^\an,\gamma)$ naturally agrees with the inverse limit of $\pi_1(S^\an,\gamma)$ over all smooth models $S$ of $k$.
\end{defn}

Denote by $\LocSys_\Q(k,\gamma)$ the category of finite-dimensional $\Q$-representations of $\pi_1((k/\C)^\an,\gamma)$, which is equivalently the category of compatible systems of $\Q$-local systems on sufficiently small models.  The category $\LocSys_\Q(k,\gamma)$ is naturally a neutral Tannakian category, whose fiber functor is the restriction to $\gamma$.  Concretely, the Tannakian group of the subcategory generated by an object $L$ is the Zariski closure of the image of the monodromy representation.  We denote the full Tannakian group of $\LocSys_\Q(k,\gamma)$ by $\Pi_1(k,\gamma)$.

We have a natural sequence of functors of neutral Tannakian categories (that is, tensor functors respecting the fiber functor)
\[\NoriMot^{\eff}(\C,\id)\to\NoriMot^{\eff}(k,\gamma)\xrightarrow{\mathscr{H}_\gamma} \LocSys_\Q(k,\gamma)\]
the first given by base-change from $\C$ to $k$ as in Proposition \ref{functionalprop} and the second the functor associated via the universal property to the representation of $\Pairs^\eff(k,\gamma)$ which sends $(X,Y,i)$ to the local system $\mathscr{H}^i_\gamma(X,Y)$ whose fiber over $s$ is $H^i(X_s,Y_s,\Q)$, for a sufficiently small model of $S$.  This is a easily checked to be a tensor functor.
\begin{thm}[{Ayoub \cite[Th\'eor\`eme 2.57]{ayoubgaloisii}, Nori}, Mostaed \cite{mostaed}]\label{exactGalois}
The resulting sequence of pro-algebraic groups
\[\Pi_1(k,\gamma)\to\G_\mot(k,\gamma)\to\G_\mot(k_0,\iota_0)\to 1\]
is exact.
\end{thm}
\begin{proof}See \cite{mostaed} for details.  The main content of the theorem is exactness in the middle.  The composition of the two middle maps is trivial since the base-change of any motive over $k_0$ to $k$ (which we henceforth call a constant motive) has trivial monodromy.  To show the image is precisely the kernel, we must lift the theorem of the fixed part to the category of motives.  Precisely, we must show that for any motive $M$ over $k$ there is a constant submotive $M_0\subset M$ whose associated local system is precisely the fixed part of $M$.  This follows for instance from a theorem of Arapura \cite[Theorem 7.1]{arapura}, which lifts the cohomology of $\mathscr{H}_\gamma(M)$ to a constant motive.
\end{proof}

\begin{cor}\label{pi1}For any Nori $(k,\gamma)$-motive $M$, the relative motivic Galois group $\bfG_\mot(M/\C,\gamma)$ is the Zariski closure of the monodromy of the Betti local system $\mathscr{H}_\gamma(M)$.
\end{cor}

\subsection{The comparison torsor}\label{noritorsor}

From now on we take $k_0=\C$ and let $k$ be a finite generated extension of $\C$, and  $\gamma$ a $(\C,\id)$-arc point of $k$.  In this section we review the construction of the torsor of comparisons between the Betti and de Rham fiber functors in the context of function fields over $\C$ as in \cite{huberms}.  We then identify the comparison torsor with the torsor $\Omega_\V$ constructed in the introduction. 

Our comparision map between Betti and de Rham cohomology will no longer be over $\C$ but instead over a larger field of germs of meromorphic functions. We therefore make the following definition:

\begin{defn}
Let $\cS$ be a complex manifold and $\gamma$ an arc point of $\cS$. We define the localization $\cO_{\cS,\gamma}$ to be 
$$\cO_{\cS,\gamma}:=\lim_{\gamma\subset U} \cO_\cS(U)$$ where the limit is over all open subsets through which the arc point factors. 

We define $k^{\an}_{\gamma}$ to be the fraction field of $\cO_{S^{\an},\gamma}$ for any model $S$ of $k$. It is immediate that this is independent of the model. 

\end{defn}

For any vertex $(X,Y,i)$ of $\Pairs^\eff(K)$, we claim there is a natural comparison 
\begin{equation}\label{comp}\phi_{X,Y,i}:H^i_{DR}(X,Y)\otimes_k k_\gamma^\an\to H^i_{\gamma}(X,Y)\otimes_\Q k_\gamma^\an.\end{equation}
By spreading out $X$ and $Y$ and possibly shrinking $S$, we may think of $X,Y$ as varieties over $S$ such that $H^i(X_t,Y_t,\Q)$ forms a local system over $S^\an$, and $H^i_{DR}(X,Y)$ is naturally the associated algebraic flat vector bundle.  The comparison \eqref{comp} is then the analytic comparison over $S^\an$ via fiberwise integration.

This naturally yields a representation of $\Pairs^\eff(k)$ and as above we therefore have a functor
\[\NoriMot(k)\to (k,\Q)_{k_\gamma^\an}\mbox{-mod},\;\;\;\; M\mapsto (H_{DR}(M),H_{\gamma}(M),\phi_M)\]
and therefore a faithful exact functor $H_{DR}:\NoriMot(k)\to k\mbox{-mod}$.  As in \cite[\S 8.4]{huberms}, there is an affine $k$-pro-scheme $\mathcal{X}$ whose points over a $k$-algebra $R$ are the isomorphisms of fiber functors
\[H_{DR}\otimes_k R\to H_\gamma\otimes_\Q R.\]
Moreover, $\mathcal{X}(k,\gamma)$ is naturally a torsor for $\bfG_\mot(k,\gamma)_k$.  Likewise, for any Nori motive $M$ over $k$, there is an affine $k$-scheme $\mathcal{X}(M,\gamma)$ of such isomorphisms of the restrictions to the tensor category $\langle M\rangle$ generated by $M$, and it is a torsor for $\bfG_\mot(M,\gamma)_k$.

Let $\mathcal{X}(k/\C,\gamma)\subset \mathcal{X}(k,\gamma)$ be the closed sub-pro-scheme of isomorphisms which restrict to the canonical comparison (fiberwise integration) on constant motives, which is naturally a torsor for $\bfG_\mot(k/\C,\gamma)$, and likewise define $\mathcal{X}(M/\C,\gamma)\subset \mathcal{X}(M,\gamma)$, which is a torsor for $\bfG_\mot(M/\C,\gamma)$.

Choose a model $S$ for $k$ such that $M$ is in the diagram category generated by pairs with models over $S$ whose cohomologies are local systems over $S^\an$. With the notation as in the introduction and taking $\V$ to be the variation of Hodge structures over $S^\an$ with underlying local system $\mathscr{H}_\gamma(M)$, there is a natural algebraic closed embedding $\mathcal{X}(M/\C,\gamma)\to \mathbb{I}_k$ by evaluating on $M$, as an isomorphism of fiber functors on $\langle M\rangle$ is determined by its value on $M$.  Moreover, analytic continuation of the canonical comparison yields a point of $\mathcal{X}(M/\C,\gamma)$ over $k^\an_\gamma$, so $\mathcal{X}(M/\C,\gamma)$ contains the germ of $\Sigma_\V$ and hence $\Omega_{\mathscr{H}_\gamma(M)}:=(\Omega_\V)_k$.  On the other hand, both $(\Omega_\V)_k$ and $\mathcal{X}(M/\C,\gamma)$ are torsors for $\bfG_\mot(M/\C,\gamma)$ by Corollary \ref{pi1}.  

Thus we deduce the following, which is essentially Nori's proof of the Grothendieck period conjecture:  

\begin{prop}\label{prop:sameperiodtorsors}
The canonical map $\Omega_{\mathscr{H}_\gamma(M)}\to \mathcal{X}(M/\C,\gamma)$ is an isomorphism.
\end{prop}

\subsection{The geometric Kontsevich--Zagier conjecture}\label{noriKZ}

We begin by defining the ring of formal periods in our setting. It shall be convenient to work with relative homology classes, so we define $H_{n,\gamma}(X,Y)$ in precisely the same way we did for cohomology, as a limit along a path up to equivalence.

\begin{defn}
The space of effective formal periods $\tilde{\bP}^{\eff}(k)$ is defined as the $\C$ - vector space generated by formal symbols $(X,Y,\omega,\ell)$ where $X$ is an algebraic variety over $k$, $Y\subset X$ is a closed subvariety, $\ell\in H^d_{DR}(X,Y)$ and $\gamma\in H_{n,\gamma}(X,Y)$ with relations given by:

\begin{enumerate}
    \item Linearity in each of $\omega,\ell$
    \item For every $f:X\ra X'$ with $f(Y)\subset Y'$ we have $$(X,Y,f^*\omega,\ell)=(X',Y',\omega',f_*\ell)$$
    \item For every triple $Z\subset Y\subset X$ $$(Y,Z,\omega,\partial\ell)=(X,Y,\delta\omega,\ell).$$
    
\end{enumerate}

We write $[X,D,\omega,\ell]$ for the image of the generator.

\end{defn}

We turn $\tilde{\bP}^{\eff}(k)$ into an algebra by setting $$[X,Y,\omega,\ell]\cdot [X',Y',\omega',\ell']:=[X\times X',Y\times Y', \omega\wedge\omega',\ell\times\ell'].$$ That multiplication is well defined is a standard check, see \cite[13.1.3]{huberms}. Finally, we define the ring of formal periods $\tilde{\bP}(k)$ to be the localization of $\tilde{\bP}^{\eff}(k)$ at $[\bG_m,\{1\},\frac{dx}{x},S^1]$.

\begin{thm}\label{thm:hubercomp}
The scheme $\Spec\tilde{\bP}(k)$ is naturally a torsor for the motivic Galois group $G_{\mot}(k,\gamma)$ base-changed to $k$. Moreover, it is naturally isomorphic to $\mathcal{X}(k,\gamma)$. 

\end{thm}

\begin{proof}

This follows identically as in \cite[13.1.4]{huberms}, using \cite[8.4.10]{huberms}.

\end{proof}

Note that there is a natural evaluation map $\textrm{ev}_k:\tilde{\bP}(k)\ra k$ given by fiber-wise integration.

Next, we define the ring of \textit{relative} formal periods. The idea is that for constant families (i.e. base-changed from $\C$), we want to identify the formal period with the actual complex number it evaluates to.

\begin{defn}
Since $\C\subset k$ there is a natural map $\tilde{\bP}(\C)\ra \tilde{\bP}(k)$. We define 
$$\tilde{\bP}(k/\C):=\tilde{\bP}(k)\times_{\tilde{\bP}(\C)}\C$$ where we view $\C$ as a $\tilde{\bP}(\C)$-algebra via the period map.
\end{defn}

\begin{thm}\label{thm:relativecomp}
The scheme $\Spec\tilde{\bP}(k/\C)$ is naturally a torsor for the relative motivic Galois group $G_{\mot}(k/\C)$. Moreover, it is naturally isomorphic to $\mathcal{X}(k/\C,\gamma)$.
\end{thm}

\begin{proof}
The first part of the theorem follows immediately by using Theorem \ref{thm:hubercomp} for $k$ and for $\C$. Indeed, there is a natural map $\Spec\tilde{\bP}(k)\ra \Spec\tilde{\bP}(\C)$ and the fiber over the point $\phi_\C$ is precisely $\Spec\tilde{\bP}(k/\C)$. 

On the other hand the relative period torsor is precisely the fiber in the $k$-period torsor over the $\C$-period torsor of the point $\phi_\C$. 

By definition $G_{\mot}(k/\C)$ is the kernel of the map $G_{\mot}(k)\ra G_{\mot}(\C)$, proving the torsor statement.

Finally, the isomorphism to $\mathcal{X}(k/\C,\gamma)$ follows from Theorem \ref{thm:hubercomp} and the fact that $\mathcal{X}(k/\C,\gamma)$ is the fiber of $\mathcal{X}(k,\gamma)$ over the canonical comparison point of $\mathcal{X}(\C,\gamma)$. 
\end{proof}

We now come to our main statement, the integrality of the relative period ring:

\begin{thm}\label{thm:KZfunctional}
The relative period ring $\tilde{\bP}(k/\C)$ is an integral domain, and the evaluation map $\textrm{ev}_k$ is an isomorphism.
\end{thm}

\begin{proof}
This is an immediate consequence of Theorem \ref{thm:relativecomp} and Proposition \ref{prop:sameperiodtorsors}, since $\Omega_{\V}$ is analytically irreducible and the Zariski closure of an analytically irreducible set is irreducible.
\end{proof}

\subsection{The geometric Andr\'e--Grothendieck period conjecture}\label{conjnew} 
Let $k\subset K\subset k_\gamma^{\an}$ be such that $K/k$ is a finitely generated extension and $\tau^*:K\to k_\gamma^\an$ a $k$-embedding.  For any models $S$ (resp. $T$) of $k$ (resp. $K$), we then obtain a rational map $f:T\to S$ and a meromorphic section $\tau :B\to T^\an$ with Zariski dense image for an open, simply connected set $B\subset S$.   The composition $\tau\circ\gamma$ is therefore an arc point of $K$.

As in \S\ref{noritorsor}, for any vertex $(X,Y,i)$ of $\Pairs^\eff(K)$, there is a natural comparison 
\[\phi_{X,Y,i}:H^i_{DR}(X,Y)\otimes_k k_\gamma^\an\to H^i_{\tau\circ\gamma}(X,Y)\otimes_\Q k_\gamma^\an\]
by pulling back the comparison \eqref{comp} over $T^\an$ along $\tau$.  We therefore obtain a functor
\[\NoriMot(K)\to (K,\Q)_{k_\gamma^\an}\mbox{-mod},\;\;\;\; M\mapsto (H_{DR,\tau}(M),H_{\tau\circ\gamma}(M),\phi_M)\]
and we define
\[K(\mathrm{periods}_{\tau^*}\;\mathrm{of}\; M)\subset k_\gamma^\an\]
to be the field of definition of $\phi_M$.  Concretely this is the field extension obtained by adjoining the pullbacks via $\tau$ of flat coordinates of algebraic sections of $H_{DR}(M)$ over $T$.

\begin{thm}[Geometric Andr\'e--Grothendieck period conjecture]\label{Andreperiod}

 \emph{}
 \newline Let $k\subset K$ be finitely generated complex fields, $\gamma$ a $(\C,\id)$-arc point of $k$, and $\tau^*: K\to k_\gamma^\an $ an embedding of $k$-extensions.  For any Nori $(K,\tau\circ\gamma)$-motive $M$ we have 
\[\trdeg_k  K(\mathrm{periods}_{\tau^*}\;\mathrm{of}\; M)\geq \dim \bfG_\mot(M/\C,\gamma).\]
\end{thm}
\begin{proof}
Note that there is a functor $MHS:\NoriMot({K},\tau\circ \gamma)\ra MHS(K)$ where $MHS(K)$ denotes the direct-limit category of admissible variations of graded-polarized integral mixed Hodge structures defined on some model of $K$. The above functor $\NoriMot(K,\tau\circ\gamma)\to(K,\Q)_{k_\gamma^\an}\mbox{-mod}$ factors through $MHS $.  

Using the above notation, for an admissible variation of graded-polarizable integral mixed Hodge structures $\cV$ over $T$, observe that since $B$ is simply connected, the map $\tau$ lifts to a map $B\ra\widetilde{T^\an}$ and therefore $\sigma_\cV\circ\tau:B\ra \Sigma_\cV$ gives a well defined map.  We have
\[\dim (\img \sigma_\cV\circ \tau)^\Zar-\dim S=\trdeg_kK(\mathrm{periods}_{\tau^*}\;\mathrm{of}\; \cV).\]
Thus, by Theorem \ref{pi1} it suffices to prove the following:

\begin{claim}\label{AndreVar}  For an admissible variation of graded-polarizable integral mixed Hodge structures $\cV$ over $T$, \begin{equation}\label{Andreineq}\dim (\img \sigma_\cV\circ \tau)^\Zar -\dim S\geq\dim\G.\end{equation}
\end{claim}

As the intersection $(\img\sigma_{\cV}\circ\tau)^\Zar\cap\Sigma_\cV$ obviously contains $\img\sigma_\cV\circ\tau$ and $\img \sigma_\cV\circ\tau$ projects to $\img \tau$ in $T$ which is Zariski dense, the claim is immediate from Theorem \ref{main}.
\end{proof}

\begin{remark}\label{rmkequal}
The geometric Andr\'e--Grothendieck period conjecture is \emph{almost} equivalent to Theorem \ref{main}, the only issue being that some variations may not come from geometry.\footnote{Whether all variations do indeed come from geometry appears to be unclear.} 

In fact, the natural generalization of the geometric Andr\'e--Grothendieck period conjecture to variations of mixed Hodge structures in the form of Claim \ref{AndreVar} \emph{is} equivalent to Theorem \ref{main}.  Indeed, the backward implication is used in the proof.  For the forward implication, let $\pi(U)$ be the projection of $U$ to $S$ and take $S'=\pi(U)^\Zar$.  Take an algebraic projection $S'\to\mathbb{A}^{\dim U}$ which is generically finite on $\pi(U)$, and take $A=\mathbb{A}^{\dim U}$.  The map $\pi(U)\to A^\an$ is generically an isomorphism, and therefore we obtain a local section $\tau$ of $S'^\an\to A^\an$ with Zariski dense image.  Applying \eqref{Andreineq} (with $(A,S')$ in the place of $(S,T)$) yields
\[\dim\G>\codim_W U\geq \dim U^\Zar-\dim U\geq\dim\G'\]
where $\G'$ is the algebraic monodromy of the restriction of $\cV$ to $S'$, so $S'$ is contained in a weak Mumford--Tate subvariety.
\end{remark}

\section{Applications}\label{elliptic}
In this section we first give a concrete example of Theorem \ref{main} for families of elliptic curves.  We then isolate some of the ideas in the example and show how the Ax--Schanuel conjecture in the form of Theorem \ref{main} allows one to formally deduce some related versions by twisting.

\subsection{Elliptic curves}
Let $S$ be a smooth irreducible variety of dimension $m$.  Let $E_1,\dots,E_n$ be non-isotrivial, pairwise non-isogenous elliptic curves over $S$ and $f_1,\dots,f_n$ sections of $E_1,\dots,E_n$ over $S$.  We therefore obtain a section $f:=(f_1,\ldots, f_n)$ of $E:=E_1\times_S\cdots\times_S E_n$.  Let $\omega_1,\dots,\omega_n$ be corresponding relative differentials, that is, sections of $H^0(\pi_*\omega_{E_i/S})$, where $\pi:E\to S$ is the projection.  We assume the $f_i$ and $\omega_i$ to be nowhere vanishing, which can always be arranged by shrinking $S$.  Finally, let $B\subset S^\an$ be an open ball over which we can trivialize the homology of $E_1,\dots,E_n$. Then by picking generators $\alpha_i,\beta_i$ of the first homology and a path $\gamma_i$ from $0$ to $f_i$, we obtain $3n$ functions by integrating the differentials along the relative homology classes $\alpha_i,\beta_i,\gamma_i$, and thus we obtain a map $F:B\ra \C^{3n}$. 

\begin{thm}

Let $T\subset \C^{3n}$ be a codimension $k$ subvariety, and suppose that $F^{-1}(T)$ contains an irreducible component $R$ of codimension $<k$. Then $R^{\Zar}\neq S$, and either two of the elliptic curves become isogenous on $R$, or at least two of the sections become torsion on $R$,  or an elliptic curve becomes isotrivial on $R$.

\end{thm}

\begin{proof}

First, note that over $S$ for each pair $(E_i,f_i)$ we have an admissible variation of graded-polarizable integral mixed Hodge structures $\cV_i=((V_i)_\Z,W_\bullet V_i,F^\bullet V_i)$ given by assigning to $s\in S$ the relative cohomology group $H^1(E_{i,s},\{0,f_i(s)\},\Z)$.  Note that this is an extension of the form
\begin{equation}0\to \Z(0)\to H^1(E_{i,s},\{0,f_i(s)\},\Z)\to H^1(E_i(x),\Z)\to0.\label{ext}\end{equation}
Let $\cV=\bigoplus_i \cV_i$. Note that the $\omega_i$ are algebraic sections of $V_\O=H^1_{DR}(E/S)$ over $S$.  By thinking of $f\in\bE$ above $s\in S$ as $f=\oplus f_i$ where $f_i:\cV_{i,s}\to V_{\C,0}$, the evaluation map $(f_i)\mapsto (f_i(\omega_i))$ gives an algebraic map $g:\Omega_\cV\to V_{\C,0}$.  Then $g\circ \sigma_\cV:\widetilde{S^\an}\to V_{\C,0}$ is the map which associates to a point $s$ together with a homotopy class of path to $s_0$ the cohomology class $([\omega_i])\in \bigoplus_i H^1(E_{i,s},\{0_s,f_i(s)\},\C)$, flatly continued to $s_0$ via the path.  In particular, using the basis $\alpha_i^\vee,\beta_i^\vee,\gamma_i^\vee$ on $V_{\C,0}$, this map agrees with $F$ on a lift of $B$.

We therefore let $W\subset\Omega_\cV$ be $W=g^{-1}(T)$, and observe that $R$ lifts to the intersection $W\cap\Sigma_\cV$.  To apply Theorem \ref{main} we must:

\begin{enumerate}
    \item[(a)] compute the algebraic monodromy group of $\cV$;
    \item[(b)] compute the dimension of $W$.
\end{enumerate}

\begin{prop}\label{maxmonodromy}
The algebraic monodromy group $\G$ of $\cV$ is $(\G_a^2\rtimes\SL_2)^n$.
\end{prop}

\begin{proof}

That $\G$ surjects onto $\SL_2^n$ follows from the fact that the elliptic curves are non-isogenous and non-isotrivial, together with the classification of weakly-special subvarieties of $X(1)^n$ (see \cite[Proposition 2.1]{Edixhoven}). Since $\SL_2$ acts irreducibly on its standard representation, we claim that it is sufficient to show that none of the algebraic monodromy groups $\G_i$ of any of the factors $\cV_i$ is $\SL_2$. Indeed, if this is the case than the unipotent radical is a sum of $(\G_a^2)^n$ which surjects to each factor and is invariant under $\SL_2^n$. Since the irreducible constituents are simply the fibers and they are mutually non-isomorphic, the claim follows.

To see that none of the $\G_i$ is $\SL_2$, we first note by Theorem \ref{MTfacts} that the algebraic monodromy group is normal in the derived subgroup of the generic Mumford-Tate group, and thus its sufficient to show that the generic Mumford-Tate group of each $\cV_i$ is maximal.
\begin{lemma}
Let $E$ be a mixed Hodge structure of the form \eqref{ext} and suppose $\gr^W_1 E$ is Mumford--Tate general.  Then the Mumford--Tate group of $E$ is $\GL_2$ if and only if the extension of mixed Hodge structures
\begin{equation}\label{ext2}0\to \Z(0)\to E\to \gr^W_1E\to0\end{equation}
is $\Q$-split.
\end{lemma}
\begin{proof}
Recall (see \S\ref{MTback}) that the Mumford--Tate group of $E$ is the stabilizer of all Hodge classes in all tensors $E^{\otimes m}\otimes (E^\vee)^{\otimes n}$.  If the Mumford--Tate group of $E$ is $\GL_2$ then there is a fixed vector in $E_\Q$ which therefore splits \eqref{ext2}, and the converse is obvious.
\end{proof}
Now it remains to note that the space of extensions \eqref{ext2} up to integral isomorphism is $(F^1\gr_1^WE)^\vee/(\gr_1^WE)_\Z^\vee\cong (\gr_1^WE)_\C/F^0\gr_1^WE+(\gr_1^WE)_\Z$ which is just the elliptic curve corresponding to $\gr_1^WE$, and the $\Q$-split points are the torsion points.
\end{proof}

We now compute the dimension of $W$.  Note that $g:\Omega_\cV\to V_{\C,0}$ is equivariant with respect to the action of $\G(\C)$.  Moreover, the class $0\neq [\omega_i]\in F^1V_i$ is not contained in $W_0V_i$ for any $i$ at any point.  Thus, the orbit of any point in the image of $g$ is an open subset of $V_{\C,0}$, and in particular of dimension $3n$.  Thus, the fibers of $g$ all have the same dimension $3n$, and $\codim_{\Omega_\cV} W=k$.  We then have $\codim_W U<\dim\G$, and it follows from Theorem \ref{main} that $R$ is contained in a proper weak Mumford--Tate subvariety.  In particular, it is not Zariksi dense, and $R$ must be contained in either:
\begin{enumerate}
    \item the locus where some $E_i$ becomes isotrivial, corresponding to the algebraic monodromy group of the restriction of $\cV_i$ being contained in $\G_a^2$;
    \item the locus where some $E_i,E_j$ for $i\neq j$ become isogenous, corresponding to the algebraic monodromy group of the restriction of $\cV_i\oplus\cV_j$ being contained in the preimage of (a conjugate of) the diagonal under $(\G_a^2\rtimes\SL_2)^2\to\SL_2^2$;
    \item the locus where some section $f_i$ of some $E_i$ is torsion, corresponding to the algebraic monodromy group of the restriction of $\cV_i$ being contained in (a lift of) $\SL_2$.
\end{enumerate}
To complete the proof we just need to show that if we are only in the last case then at least two sections become torsion.  Assume that only one section (without loss of generality $f_1$) becomes torsion.  Consider the variation $\cV':=\gr^W_1\cV_1\oplus\bigoplus_{i>1}\cV_i$, which is a quotient of $\cV$.  Let $V_{\C,0}\to V'_{\C,0}$ be the corresponding quotient and $T'$ the image of $T$.  Now $\codim T'\geq k-1$ and $\codim_{R^\Zar} R<k-1$, so applying the same analysis as above we obtain a contradiction.

\end{proof}

\subsection{Ax--Lindemann for abelian differentials} In this section we prove a recent conjecture of Klingler--Lerer \cite{klinglerlerer}.  We first briefly recall strata of abelian differentials.

Let $g> 0$ be an integer, $\alpha$ a partition of $2g-2$, and $S=S_\alpha$ the moduli space of pairs $(C,\omega)$ where $C$ is a genus $g$ curve and $\omega$ is a regular 1-form on $C$ whose zero divisor $Z(\omega)$ has type $\alpha$, meaning it is of the form $\sum \alpha_ip_i$ for distinct points $p_i$ on $C$.  There is a natural variation of mixed Hodge structures over $S_\alpha$ whose fiber over $(C,\omega)$ is the relative cohomology group $H^1(C^\an,Z(\omega),\Z)$.  

Fixing a basepoint $(C_0,\omega_0)$, let $\pi:\widetilde{S^\an}\to S$ be the universal cover.  The map $\phi:\widetilde{S^\an}\to V_0:=H^1(C_0,\omega_0,\C)$ mapping $(C,\omega)$ to the image of the class $[\omega]\in H^1(C,\omega,\C)$ under the flat trivialization is a local isomorphism by a theorem of Veech \cite[Thm. 7.15]{veech}.

Following \cite{klinglerlerer}, we say an (irreducible) algebraic subvariety $W\subset S$ is bialgebraic if the Zariski closure of $\phi(W_0)$ in $V_0$ has dimension $\dim W$ for some (hence any) component $W_0$ of $\pi^{-1}(W)$.  We likewise say $W\subset V_0$ is bialgebraic if the Zariski closure of $\pi(W_0)$ has dimension $\dim W$ for some (hence any) component $W_0$ of $\phi^{-1}(W)$.  The following is the Ax--Lindemann conjecture of \cite{klinglerlerer}.

\begin{thm}
For any algebraic subvariety $W\subset V_0$, the Zariski closure of $\pi(W_0)$ is bialgebraic for any component $W_0$ of $\phi^{-1}(W)$.
\end{thm}
\begin{proof}
Let $X$ be the Zariski closure of $\pi(W_0)$ and let $\V$ be the above variation of mixed Hodge structures restricted to $X$.  As in the previous section, there is a natural algebraic evaluation map $r:\Omega_\V\to V_0$ through which $\phi|_X=r\circ\sigma_\V$ factors. 

We claim that $r(\Omega_\V)$ and $\G(\C)W$ have the same Zariski closure.  Indeed, $Y:=\overline{\G(\C)W}$ is certainly contained in $\overline{r(\Omega_\V)}$.  On the other hand, $Y $ is $\G(\C)$-invariant so the pullback to $\widetilde{X^\an}$ descends to an algebraic subvariety of $X$ by definable Chow \ref{definechow} and contains $\pi(W_0)$, hence it must be all of $X$.  Thus, $\phi(\widetilde{X^\an})$ is contained in $Y$, as therefore is its Zariski closure $\overline{r(\Omega_\V)}$, so we have the inclusion in the reverse directions.

As $\phi(\widetilde{X^\an})\subset Y$, for $X$ to be bialgebraic it suffices to show $\dim X\geq\dim Y$.  Let $W'$ be the pullback of $W$ to $\Omega_\V$.  The dimension of $W'$ is $\dim F+\dim W$ where $F$ is the generic fiber of $r:\Omega_\V\to V_0$ over $W$.  By the above $r(\Omega_\V)$ and $\G(\C)W$ have the same Zariski closure, and since $r$ is $\G(\C)$-equivariant it follows that the generic fiber dimension of $r$ over its image is equal to the generic fiber dimension over $W$.  Thus, 
\begin{align*}
    \dim W'&= \dim \Omega_\V-\dim Y+\dim W\\
    &=\dim \G+\dim X-\dim Y+\dim W.
\end{align*}
Now $\pi(W_0)$ is Zariski dense in $X$ but also lifts to the intersection of $W'$ with a leaf, so by Theorem \ref{main} we must have
\[\G\leq \dim W'-\dim W=\dim\G+\dim X-\dim Y\]
and therefore 
$\dim X\geq \dim Y$ as desired.
\end{proof}
\subsection{Twisting by the period torsor}\label{secotherAS}In this section we explain formally how one may deduce many of the previous Ax-Schanuel theorems from Theorem \ref{main}.

In applications, one often has a variety $M$ with an algebraic (left) $\G(\C)$-action and an equivariant algebraic map $g:\Omega_\cV\to M$, as in the last subsection.  In this case Theorem \ref{main} is readily applied.

Another common situation is to have an algebraic variety $\mathbb{P}\to S$ over $S$, an algebraic variety $P$ equipped with a (left) $\G(\C)$-action, and a $\G(\C)$-equivariant map
\begin{equation}\langle\;,\;\rangle:\Omega_\cV\times_S\mathbb{P}\to P \label{twisted}\notag \end{equation}
which we call a \emph{twisting map}.  Such a map yields a map $\langle \sigma_\cV,\;\rangle:\widetilde {S^\an}\times_{S^\an}\mathbb{P}^\an\to P^\an$ on the base-change to the universal cover.  Note that these two setups are equivalent, as we may take $\langle\;,\;\rangle$ to be the $\G(\C)$-equivariant map
\[g_{\langle\;,\;\rangle}:\Omega_\cV\to \Hom_S(\mathbb{P},P_S)\]
where $P_S=P\times S$, and in the other direction we may take $\mathbb{P}=S$ and ${\langle\;,\;\rangle}_g=g$.
\medskip

Examples of twisting maps include:
\begin{itemize}
    \item $\mathbb{P}=\mathbb{V}$ and $P=V_{\C,0}$ and $\langle\;,\;\rangle$ the obvious evaluation.  The map $\langle \sigma_\cV,\;\rangle$ is then the flat trivialization.  
    \item $\mathbb{P}=$ the relative flag variety of $V_\O$ for which the Hodge filtration yields a section, $P=$ the flag variety of $V_{\C,0}$ containing the relevant period domain (that is, its dual), and $\langle f, F^\bullet V_s \rangle=f(F^\bullet V_s)$.  The map $\langle \sigma_\cV,\;\rangle$ is then the period map.
    \item For any artinian ring $A$ and any twisting map $\langle\;,\;\rangle:\Omega_\cV\times_S\mathbb{P}\to P$ we get a map on $A$-jet spaces 
    \[J_A\Omega_\cV\times_{J_AS}J_A\mathbb{P}\to J_AP.\]
    The horizontal jets yield a natural subspace $\Omega_\cV\times_SJ_AS\subset J_A\Omega_\cV $ which is preserved by the $\G(\C)$ action.  We therefore obtain a twisting map
    $\langle\;,\;\rangle_A:\Omega_\cV\times_SJ_A\mathbb{P}\to J_AP$ and the map $\langle\sigma_\cV,\;\rangle_A$ is then the map on jet spaces induced by $\langle \sigma_\cV,\;\rangle$.  In this way we may access transcendence statements for the derivatives of $\langle \sigma_\cV,\;\rangle$, as in \cite{MPT}.
    \item By taking $\mathbb{P}$ to be $S\times X$ and $P=\Omega_{\cV}\times X$ for a variety $X$, we obtain the Ax-Schanuel result ``in families", or ``relative Ax-Schanuel" as it has been called in the literature.

\end{itemize}
Given a twisting map we define
\[\mu:=\langle \;,\;\rangle\times \pi_2 :\Omega_\cV\times_S\mathbb{P}\to P\times \mathbb{P}.\]
We say the twisting map is \emph{balanced} if the fibers of $\mu$ all have the same dimension.  Note that the fiber over $(p',p)$ is identified with the stabilizer $\Stab_{\G(\C)}(p)$.  In practice, given a twisting map we can always assume it is balanced by passing to a Zariski open subset.
\begin{prop} Let $\langle\;,\;\rangle$ be a balanced twisting map as above and let $\Delta\subset\img\mu$ be the image of $\Sigma_\cV\times_{S^\an}\mathbb{P}^\an$ under $\mu^\an$.  Let $W\subset \img\mu$ be an algebraic variety and $U$ a component of $W^\an\cap\Delta$ such that
\[\codim_\Delta U<\codim_{\img \mu} W.\]
Then the projection of $U$ to $S^\an$ is contained in a weak Mumford--Tate subvariety.
\end{prop}
\begin{proof}Let $k$ be the dimension of the fibers of $\mu$.  First, $U$ naturally lifts to $\Sigma_\cV\times_{S^\an}\mathbb{P}^\an\subset \Omega_\cV^\an\times_{S^\an} \mathbb{P}^\an$.  Call this lift $U'$.  Let $W'$ be the component of the preimage of $W$ under $\mu$ which contains $U'$, and let $W''\subset \Omega_\cV$ (resp. $U''\subset \Sigma_\cV$) be the image of $W'$ (resp. $U'$) under the first projection.  Clearly $W''\cap \Sigma_\cV$ contains $U''$ and we also have that $\dim W'=\dim W+k$.  The fibers of $W'\to W''$ are the intersections of $W$ with subvarieties of the form $p'\times\mathbb{P}_t$, and these are the same fibers as $U'\to U''$ over $U''$.  Up to replacing $W$ with an algebraic subvariety for which the generic fiber of $W'\to W''$ has the same size as the generic fiber of $U'\to U''$ (and without changing $U$), we therefore have
\begin{align*}
    \codim_{W''} U''&=\codim_{W'} U'
    \\
    &=\codim_W U+k\\
    &=\codim_{\Delta}U-\codim_{\img\mu}W+(k+\dim\img\mu-\dim\Delta)\\
    &<\dim\G.
\end{align*}
Applying Theorem \ref{main}, the result follows.
\end{proof}
Note that in the context of the proposition, $\img\mu$ is the Zariski closure of $\Delta$ in $P\times\mathbb{P}$.

\bibliography{NoriRewrite}
\bibliographystyle{plain}

\end{document}